\documentclass[11pt,twoside,a4]{article} 

\usepackage{amssymb}
\usepackage{graphicx}
\usepackage{amscd}
\usepackage{color}
\usepackage{float}
\usepackage{graphics}
\usepackage{amsmath}
\usepackage{amssymb}
\usepackage{mathrsfs}
\usepackage{amsthm}
\usepackage{amsfonts}
\usepackage{latexsym}
\usepackage{epsf}
\usepackage[]{hyperref}
\usepackage{fancyhdr}

\makeatletter
\newenvironment{proof*}[1][\proofname]{\par
  \pushQED{\qed}%
  \normalfont \partopsep=\z@skip \topsep=\z@skip
  \trivlist
  \item[\hskip\labelsep
        \itshape
    #1\@addpunct{.}]\ignorespaces
}{%
  \popQED\endtrivlist\@endpefalse
}
\makeatother

\date{}

\begin{document} 

\centerline{\bf Applied Mathematical Sciences, Vol. x, 2015, no. xx, xxx - xxx}

\centerline{\bf HIKARI Ltd, \ www.m-hikari.com}

\centerline{} 

\centerline{} 

\centerline {\Large{\bf On Linear Recursive Sequences with Coefficients}} 

\centerline{\Large{\bf in Arithmetic-Geometric Progressions\footnote{to appear by April 2015}}} 


\centerline{} 

\centerline{\bf {Jerico B. Bacani and Julius Fergy T. Rabago}} 

\centerline{} 

\centerline{Department of Mathematics and Computer Science} 

\centerline{College of Science} 

\centerline{University of the Philippines Baguio} 

\centerline{Baguio City 2600, Philippines} 

\centerline{jicderivative@yahoo.com, jfrabago@gmail.com}

\newtheorem{theorem}{Theorem}[section]
\newtheorem{lemma}[theorem]{Lemma}
\newtheorem{corollary}[theorem]{Corollary}
\newtheorem{proposition}[theorem]{Proposition}
\newtheorem{definition}[theorem]{Definition} 
\newtheorem{example}[theorem]{Example}
\newtheorem{remark}[theorem]{Remark}

\centerline{}

{\footnotesize Copyright $\copyright$ 2014 Jerico B. Bacani and Julius Fergy T. Rabago. This is an open access article distributed under the Creative Commons Attribution License, which permits unrestricted use, distribution, and reproduction in any medium, provided the original work is properly cited.}

\begin{abstract} 
We present a certain generalization of a recent result of M. I. C\^{i}rnu on linear recurrence relations with coefficient in progressions \cite{cirnu}. We provide some interesting examples related to some well-known integer sequences, such as Fibonacci sequence, Pell sequence, Jacobsthal sequence, and the Balancing sequence of numbers. The paper also provides several approaches in solving the linear recurrence relation under consideration. We end the paper by giving out an open problem.\\ 
\end{abstract} 



\section{Introduction}
A \emph{$k^{th}$-order linear recurrence relation} for a sequence $\{x_n \}_{n=0}^{+\infty}$ has the form
$$x_n = a_1 x_{n-1} +a_2 x_{n-2} + a_3 x_{n-3} + \ldots + a_k x_{n-k} + f_n \qquad (n \geq k),$$
where $a_i's$ are constants, and $f_k, f_{k+1}, f_{k+2}, \ldots$ is some given sequence. This linear recurrence relation is homogeneous if $f_n = 0$ for all $n$; otherwise, it is non-homogeneous. 

The sequence $\{x_n \}_{n=0}^{+\infty}$ that satisfies the relation above is called \emph{linear recurrence sequence}. It is an interesting topic in number theory because of its vast applications in science and mathematics. The simplest type of recurrence sequence is the \emph{arithmetic progression}, popularly known as \emph{arithmetic sequence}. It is a number sequence in which every term except the first, say $a$, is obtained by adding the preceding term a fixed number $d$, called the \emph{common difference}. The set $\mathbb{N}$ of natural numbers is a very good example for this. If $a_n$ denotes the $n^{th}$ term of an arithmetic sequence, then we have 
	\begin{equation}\label{arith}a_1=a, a_n=a_{n-1}+d \quad(n\geq2).\end{equation}
An explicit formula for $a_n$ is given by
	\[
		a_n=a+(n-1)d\quad(n\geq2).
	\]
The number sequence defined by the relation \eqref{arith} is an example of a linear recurrence sequence of order one. The sum $S_n$ of the terms in this progression is given by
	\[
		S_n = \frac{n}{2}[2a+ (n-1)d] \quad(n\geq1).
	\]
	
Another type of number sequences is the \emph{geometric progression}. It is a number sequence in which every term except the first is obtained by multiplying the previous term by a constant number $r$, called the \emph{common ratio.} The sequence $3,9,27,81,\ldots$ is an example of a geometric sequence with a common ratio 3. If $a_n$ denotes the $n^{th}$ term of the sequence with first term $a$ and common ratio $r,$ then $a_n$ is defined recursively as 
	\begin{equation}\label{geo}a_1=a, \quad a_n=a_{n-1}\cdot r \quad(n\geq2).\end{equation}
An explicit formula for $a_n$ is given by
	\[
		a_n=a\cdot r^{n-1} \quad(n\geq1).
	\]
The sum $S_n$ is given by
	\[
		S_n = a\frac{1 - r^n}{1-r},\;r\neq1 \quad(n\geq1).
	\]
If $|r|<1$, then we have 
	\begin{equation}\label{geolim}\lim_{n\rightarrow\infty}S_n=a\sum_{n=0}^{\infty}r^n= \frac{a}{1-r}.\end{equation}

Recently, various generalizations of arithmetic and geometric progressions were offered by several authors. In \cite{zhang}, X. Zhang and Y. Zhang introduced the concept of arithmetic progression with two common differences, and in \cite{ding}, X. Zhang and two others generalized this sequence by injecting a \emph{period} with alternate common differences. These concepts were then extended by A. A. K. Majumdar to geometric progressions \cite{majumdar}, in which an alternative approach to some results in \cite{ding} was also presented. Further extensions of these concepts are found in \cite{rabago1}, \cite{rabago2}, \cite{rabago3}, and \cite{rabago4}.   

In this paper we shall provide a generalization of a recent result of M. I. C\^{i}rnu on linear recurrence relations with coefficients in progressions \cite{cirnu}. We give some interesting examples related to some well-known sequences (e.g. Fibonacci sequence, Pell sequence, Jacobsthal sequence, and the Balancing number sequence). The results are elementary; however, the present study provides the readers new properties of some special types of recurrence sequences and several procedures in dealing with similar types of problems. Finally, we end the paper with an open problem.
\section{Main result}

In \cite{cirnu}, C\^{i}rnu solved the linear recurrence relation 
	\begin{equation}\label{relation}
		x_{n+1}=a_0x_{n+1}+a_1x_{n-1}+\cdots+a_{n-1}x_1+a_nx_0\quad (n\geq0),
	\end{equation}
where its coefficients $a_i's$ form an arithmetic (or geometric) progression. In short, he provided explicit formulas for $x_n$ to the following recurrence relations:

\begin{enumerate}
	\item[(i)] $x_{n+1}=ax_n+(a+d)x_{n-1}+\cdots+(a+(n-1)d)x_1+ (a+nd)x_0$; and
	\item[(ii)] $x_{n+1}=ax_n+aqx_{n-1}+\cdots+aq^{n-1}x_1+ aq^nx_0$,
\end{enumerate}
with initial data $x_0$. These two sequences defined by the two relations above were considered separately in \cite{cirnu}. In this paper, however, we dealt with the two sequences simultaneously. In other words, we considered the convolved sequence 
	\begin{equation}\label{ag}
		a,\;\; (a+d)r, \;\; (a+2d)r^2,\;\;\ldots,\; \;\;(a+(n-1)d)r^{n-1},\;\; (a+nd)r^n,
	\end{equation}
rather than dealing with (i) and (ii) separately. 

The sequence \eqref{ag} is called an \emph{arithmetic-geometric progression}. Though its form appears to be very simple, it has not gained much attention of mathematicians unlike the well-known Fibonacci sequence, Pell sequence, Jacobsthal sequence,and Balancing number sequence \cite{behera}. In fact, not much information is available about this sequence except for the formula for the sum $S_n$ of its first $n$ terms; that is,
	\[
		S_n=\sum_{k=0}^n\left( a+kd\right)r^k=\frac{a-\left( a+nd\right)r^{n+1}}{1-r}+dr\frac{1-r^n}{(1-r)^2}.
	\]
It is easy to verify that for $r \in (-1,1)$, 
	\[
		S_n \longrightarrow \frac{a}{1-r}+\frac{dr}{(1-r)^2}\quad \text{as}\quad n\longrightarrow \infty.
	\] 
But if $r\in \mathbb{R}\setminus(-1,1)$ and $n$ tends to infinity, $S_n$ diverges.
Perhaps, the sequence \eqref{ag} does not posses fascinating properties that the Fibonacci sequence does. However, as we shall see later, the sequence \eqref{ag} is somehow related to some well-known recurrence sequences of order two. 

By interchanging the operations of addition and multiplication by a constant, we may also define analogously a sequence that we call \emph{geometric-arithmetic progression}. This sequence is of the form 
	 \begin{equation}\label{ga}
		a,\;\; ar+d, \;\; ar^2+2d,\;\;\ldots,\; \;\;ar^{n-1}+(n-1)d,\;\; ar^n+nd.
	\end{equation}
It can be verified that the sum of the first $n$ terms of the sequence \eqref{ga} is given by
	\begin{equation}\label{sumga} S_n=a\frac{1-r^n}{1-r}+\frac{n(n-1)}{2}d.\end{equation} 
Now, combining the idea behind the usual arithmetic progression generated by \eqref{arith} and the concept of arithmetic-geometric progression (resp. geometric-arithmetic progression), we come up with the following recurrence relations of order one:
\begin{align}
	&a_0=a, \quad a_n = a_{n-1} r + d \quad \quad (n\geq1);\label{r1}\\
	&a_0=a, \quad a_n = (a_{n-1} + d)r \quad \,\,(n\geq1)\label{r2}.
\end{align}
For $r\neq 1$, the corresponding formulas of $n^{th}$ term for \eqref{r1} and \eqref{r2}, are as follows:
\begin{align}
	&a_n=ar^n+d\frac{1-r^n}{1-r};\label{a1}\\
	&a_n=ar^n+dr\frac{1-r^n}{1-r}.\label{a2}
\end{align}
For $r=1$, the recurrence relations \eqref{r1} and \eqref{r2} coincide with the usual arithmetic progression \eqref{arith}.

In this section, we consider the relation \eqref{relation}, where $a_n=(a+nd)r^n$; that is, we study the following recurrence relation:
	\begin{equation}\label{agrelation}
		x_{n+1}=ax_n + (a+d)r x_{n-1}+(a+2d)r^2 x_{n-2}+\cdots+(a+nd)r^nx_0\quad(n\geq 0).
	\end{equation}
Surprisingly, \eqref{agrelation} is related to a Horadam-like sequence since \eqref{agrelation} can be reduced to a second-order linear recurrence relation as we shall see in the proof of this result.

\begin{theorem}\label{cirnuthm}
	The sequence $\{x_n\}$ satisfies the recurrence relation \eqref{agrelation} with coefficients in arithmetic-geometric progression if and only if the sequence $\{x_n\}$ satisfies the generalized (second-order) Fibonacci sequence with the Binet fomula given by
		\begin{equation}\label{agsol}
			x_n=\frac{x_0}{\lambda_1-\lambda_2}\left[\left(B-a\lambda_2\right)\lambda_1^{n-1}-\left(B-a\lambda_1\right)\lambda_2^{n-1}\right]\quad (n\geq 1),
		\end{equation}
where  $B=a^2+(a+d)r$ and $\lambda_{1,2}=\frac{1}{2}\left(a+2r\pm\sqrt{a^2-4r(r-d-1)}\right)$. 
\end{theorem}
\begin{proof}
We prove the theorem by reducing \eqref{agrelation} to a second order linear recurrence relation of order two. We suppose that the sequence $\{x_n \}$ satisfy the recurrence relation \eqref{agrelation}. Then, we have $x_1=ax_0,\;x_2=a^2x_0+(a+d)rx_0,$ and   	
	\begin{eqnarray}
	x_{n+1} - rx_n&=&\sum_{k=0}^n\left( a+kd\right)r^kx_{n-k}-r\sum_{k=0}^{n-1}\left( a+kd\right)r^kx_{(n-1)-k}\\\nonumber
			   &=& ax_n+drx_{n-1}+dr^2x_{n-2}+\cdots+dr^nx_0;	\\
	x_n - rx_{n-1}&=&\sum_{k=0}^{n-1}\left( a+kd\right)r^kx_{(n-1)-k}-r\sum_{k=0}^{n-2}\left( a+kd\right)r^kx_{(n-2)-k}\\\nonumber
			  &=& ax_{n-1}+drx_{n-2}+dr^2x_{n-3}+\cdots+dr^{n-1}x_0.
	\end{eqnarray}
Hence, 
	\[
		(x_{n+1}-rx_n) -r(x_n-rx_{n-1})=ax_n+drx_{n-1}-arx_{n-1},
	\]
or equivalently,
	\begin{equation}\label{genfib}
		x_{n+1}=Px_n+Qx_{n-1},
	\end{equation}
where $P=a+2r$ and $Q=-(r^2+(a-d)r)$. We recognized that \eqref{genfib} portrays a Horadam-like sequence (cf. \cite{horadam}).  Hence, we reduced the recurrence relation \eqref{agrelation} of order $n$ to a second-order linear recurrence relation \eqref{genfib}. In fact, \eqref{genfib} can be further reduced to a linear recurrence relation of order one which will be shown later. 

Note that there are several methods known in solving linear recurrences. So we first provide several approaches in obtaining the general solution to \eqref{genfib} so as to help the readers get familiarized with solving similar problems.

\subsection*{Approach 1 (Using a discrete function $\lambda^n, \lambda\in\mathbb{R}, n\in \mathbb{N}$) }
Let $x_n=\lambda^n, \lambda\neq 0$. So $\lambda^{n+1}=P\lambda^n+Q\lambda^{n-1}$, which is equivalent to $\lambda^2-P\lambda-Q=0$. This characteristic equation of \eqref{genfib} has the solutions
	\[
		\lambda_{1,2}=\frac{P\pm\sqrt{P^2+4Q}}{2}.
	\]
Hence, \eqref{genfib} has the general solution $x_n=C_1\lambda_1^n+C_2\lambda_2^n$, where $C_1,C_2 \in \mathbb{R}$ subject to the initial conditions $x_1=C_1\lambda_1+C_2\lambda_2=ax_0$ and $x_2=C_1\lambda_1^2+C_2\lambda_2^2=a^2x_0+(a+d)rx_0$. Solving for $C_{1,2}$ we obtain $$C_{1,2}=\pm x_0\left(\frac{a^2+(a+d)r-a\lambda_{2,1}}{\lambda_{1,2}(\lambda_1-\lambda_2)}\right).$$
Thus the solutions of the recurrence relation \eqref{agrelation} are given by the formula \eqref{agsol}.
\begin{remark}\label{rem1}
	The recurrence relation \eqref{genfib} fails to hold for $n=0$. Hence, the initial conditions are $x_1$ and $x_2$ instead of $x_0$ and $x_1$.
\end{remark}
\subsection*{Approach 2 (Via reduction to order one)}
Let $\lambda_{1,2}$ be the roots of the quadratic equation $x^2-Px-Q=0$, where $P=a+2r$ and $Q=-(r^2+(a-d)r)$. Evidently, since $\lambda_1+\lambda_2=P$ and $\lambda_1\lambda_2=-Q$, we have $x_{n+1}=(\lambda_1+\lambda_2)x_n-\lambda_1\lambda_2x_{n-1}$, or equivalently,
	\begin{equation}\label{reduction}
		x_{n+1} -\lambda_1x_n=\lambda_2(x_n -\lambda_1x_{n-1}). 
	\end{equation}
Note that the sequence $\{x_{n+1} -\lambda_1x_n\}_{n=1}^{\infty}$ can be viewed as a geometric progression with $\lambda_2$ as the common ratio. Hence, by iterating $n$, we get
	\[
		x_{n+1} -\lambda_1x_n=\lambda_2^{n-1}(x_2 -\lambda_1x_1).
	\]
Dividing both sides by $\lambda_2^n$, and doing algebraic manipulations, we obtain
	\begin{equation}\label{yn}
		y_{n+1}=\frac{\lambda_1}{\lambda_2}y_n+\frac{x_2-\lambda_1x_1}{\lambda_2},
	\end{equation}
 where $y_n\colon\hskip-8pt=x_n/\lambda_2^n$. Suppose $P^2+4Q>0$. Then it follows that $\lambda_1>\lambda_2$. Letting $\rho=\lambda_1/\lambda_2$ and $\delta=(x_2-\lambda_1x_1)/\lambda_2$, we can express \eqref{yn} as
	$y_{n+1}=\rho y_n + \delta$, which is similar to the linear recurrence equation \eqref{r1} of order one.
Hence, in view of \eqref{a1} with $\rho\neq 1$ (together with Remark \eqref{rem1}) and by replacing $n+1$ by $n$, we get 
	\[
		\frac{x_n}{\lambda_2^n}=x_1\left(\frac{\lambda_1}{\lambda_2}\right)^n+\frac{x_2-\lambda_1x_1}{\lambda_2}\left[\frac{1-\left(\frac{\lambda_1}{\lambda_2}\right)^n}{1-\frac{\lambda_1}{\lambda_2}}\right].
	\]	
This equation yields
	\begin{equation}\label{eks}x_n=x_1\lambda_1^n+\frac{x_2-\lambda_1x_1}{\lambda_1-\lambda_2}\left(\lambda_1^n-\lambda_2^n\right).\end{equation}
We notice that for $n=1$, we get $x_1=x_2$. Replacing $n$ by $n-1$ on the LHS of \eqref{eks}, we obtain
	\[
		x_n=x_1\lambda_1^{n-1}+\frac{x_2-\lambda_1x_1}{\lambda_1-\lambda_2}\left(\lambda_1^{n-1}-\lambda_2^{n-1}\right).
	\] 
The above equation can be simplified into this desired formula:
	\[
		x_n=\frac{x_0}{\lambda_1-\lambda_2}\left[\left(B-a\lambda_2\right)\lambda_1^{n-1}-\left(B-a\lambda_1\right)\lambda_2^{n-1}\right].
	\]
\subsection*{Approach 3 (Using generating functions)}
Let $X(t)$ be the generating function for $\{x_n\}$. By considering Remark \eqref{rem1} and re-indexing,  we write $X(t)=\sum_{n=1}^{\infty}x_nt^{n-1}$. On the other hand, we multiply the equation \eqref{genfib} by $t^{n-1}$ and sum up the terms over $n\geq2$ to get
	\begin{equation}\label{ekst}\sum_{n=2}^{\infty}x_{n+1} t^{n-1}=P\sum_{n=2}^{\infty}x_n t^{n-1}+Q\sum_{n=2}^{\infty}x_{n-1} t^{n-1}.\end{equation}
After doing some algebraic manipulations, we get an equivalent form of \eqref{ekst}:
	\[
		\frac{1}{t}\left(X(t) -x_1-x_2t\right)=PX(t)-Px_1+QtX(t).
	\] 
Solving for $X(t)$, we obtain
	\[
		X(t) = \frac{(Px_1-x_2)t-x_1}{Qt^2+Pt-1}.
	\]
Thus, 
	\begin{align*}
		\sum_{n=1}^{\infty}x_nt^{n-1}&=\frac{(Px_1-x_2)t-x_1}{Qt^2+Pt-1}
				=\frac{(Px_1-x_2)t-x_1}{Q\left(t+\frac{P+\sqrt{P^2+4Q}}{2Q}\right)\left(t+\frac{P-\sqrt{P^2+4Q}}{2Q}\right)}\\
			&=\frac{(Px_1-x_2)t-x_1}{\lambda_2\left(t+\frac{\lambda_1}{Q}\right)\cdot \lambda_1\left(t+\frac{\lambda_2}{Q}\right)}
				=\frac{(Px_1-x_2)t-x_1}{\left(1-\lambda_2t\right)\left(1-\lambda_1t\right)}\\
			&=\frac{1}{\lambda_1-\lambda_2}\left[\frac{\lambda_1x_1-(Px_1-x_2)}{1-\lambda_1t}-\frac{\lambda_2x_1-(Px_1-x_2)}{1-\lambda_2t}\right]\\
			&=\frac{1}{\lambda_1-\lambda_2}\left(\frac{x_2-\lambda_1x_1}{1-\lambda_2t}-\frac{x_2-\lambda_2x_1}{1-\lambda_1t}\right).
	\end{align*}
Applying \eqref{geolim}, we get
\begin{align*}
		\sum_{n=1}^{\infty}x_nt^{n-1}
			&=\frac{1}{\lambda_1-\lambda_2}\left[ (x_2-\lambda_2x_1)\sum_{n=1}^{\infty}\lambda_1^{n-1}t^{n-1}-(x_2-\lambda_1x_1)\sum_{n=1}^{\infty}\lambda_2^{n-1}t^{n-1}\right]\\
			&=\sum_{n=1}^{\infty}\left[\frac{1}{\lambda_1-\lambda_2}\left( (x_2-\lambda_2x_1)\lambda_1^{n-1}-(x_2-\lambda_1x_1)\lambda_2^{n-1}\right)\right]t^{n-1}.
	\end{align*}
	By taking $x_1=ax_0$ and $x_2=(a^2+(a+d)r)x_0=Bx_0$ and dropping down the summation symbol, we obtain the desired result. 
\subsection*{Approach 4 (Using induction on $n$)}
We claim that \eqref{agsol} is the solution to \eqref{genfib}. First, we note that the formula \eqref{agsol} holds true for $n=1,2$. Next, we assume that \eqref{agsol} is a solution to \eqref{agrelation} for some integers $k$ and $k-1$ where $k\leq n$ for some fixed $n\geq 2$. Hence, from \eqref{genfib}, we have 
	\begin{align*}
	x_{k+1}	&=Px_k+Qx_{k-1}\\
			&=\frac{x_0}{\lambda_1-\lambda_2}\left[\left(B-a\lambda_2\right)\lambda_1^{k-2}(P\lambda_1+Q)
					-\left(B-a\lambda_1\right)\lambda_2^{k-2}(P+Q\lambda_2)\right].
	\end{align*} 
Since $\lambda_{1,2}$ are the roots of the characteristic equation $\lambda^2-P\lambda-Q=0$ of \eqref{genfib}, then $\lambda_{1,2}^2=P\lambda_{1,2}+Q$. Thus, we get
	\begin{align*}
		x_{k+1}  
			     &=\frac{x_0}{\lambda_1-\lambda_2}\left[\left(B-a\lambda_2\right)\lambda_1^{k}-\left(B-a\lambda_1\right)\lambda_2^{k}\right].
	\end{align*}
By principle of mathematical induction, we see that \eqref{agsol} is the solution to \eqref{genfib}.

Now, to complete our proof, we show that \eqref{agsol} also satisfies the relation \eqref{agrelation}. We again proceed by induction. For $n=1,2,3$ we have the following:
	\begin{align*}
	x_1	&=\frac{x_0}{\lambda_1-\lambda_2}\left[(B-a\lambda_2)-(B-a\lambda_1)\right]=ax_0,\\
	x_2	&=\frac{x_0}{\lambda_1-\lambda_2}\left[(B-a\lambda_2)\lambda_1-(B-a\lambda_1)\lambda_2\right]\\
	     	&=\frac{x_0}{\lambda_1-\lambda_2}\left[B\lambda_1-aQ-(B\lambda_2-aQ)\right]\\
	      	&=x_0B=a^2x_0+(a+d)rx_0\\
		&=ax_1+(a+d)rx_0,\\
	x_3	&=\frac{x_0}{\lambda_1-\lambda_2}\left[(B-a\lambda_2)\lambda_1^2-(B-a\lambda_1)\lambda_2^2\right]\\
	     	&=\frac{x_0}{\lambda_1-\lambda_2}\left[B(\lambda_1^2-\lambda_2^2)-a\lambda_1\lambda_2(\lambda_1-\lambda_2)\right]\\
		&=x_0\left(BP+aQ\right)=x_0[B(a+2r)-a(r^2+ar-dr)]\\
	      	&=x_0(aB+2a^2r+2ar^2+2dr^2-ar^2-a^2r+adr)\\
	      	&=x_0(aB+a(a+d)r+(a+2d)r^2)\\
		&=ax_2+(a+d)rx_1+(a+2d)r^2x_0.
	\end{align*}
Now, we suppose that \eqref{agrelation} is true for all $k\leq n$ where $n\geq 2$ is fixed, i.e.,	
	\[
		x_{k+1}=\sum_{l=0}^k(a+(k-l)d)r^{k-l} x_l\quad(k\leq n).
	\]
From \eqref{genfib}, together with the above equation, we obtain
	\begin{align*}
		x_{k+2}	&=Px_{k+1}+Qx_k\\
				&=P\sum_{l=0}^k(a+(k-l)d)r^{k-l} x_l+Q\sum_{l=0}^{k-1}(a+(k-1-l)d)r^{k-1-l} x_l\\
				&=P\sum_{l=0}^k(a+(k-l)d)r^{k-l} x_l+Q\sum_{l=1}^k(a+(k-l)d)r^{k-l} x_{l-1}.
	\end{align*}  
Now, with $P=a+2r$ and $Q=-r^2-ar+dr$, we get 
	\begin{align*}
		x_{k+2}	&=\sum_{l=2}^k(a+(k-l)d)r^{k-l} (Px_l+Qx_{l-1})\\
				&\hspace{.75in}+(a+2r)(a+kd)r^k x_0+(a+2r)(a+(k-1)d)r^{k-1} x_1\\
				&\hspace{1.5in}-(r^2+ar-dr)(a+(k-1)d)r^{k-1} x_0.
	\end{align*}  
Letting $A_k=(a+kd)r^k$, we obtain 
		\begin{align*}
		x_{k+2}	&=\sum_{l=2}^kA_{k-l} x_{l+1}+ a x_0 A_k + [2(a+kd)-(a+(k-1)d)]r^{k+1}x_0 \\
				&\quad\quad+ a^2x_0 A_{k-1} + 2arA_{k-1}x_0 -arA_{k-1}x_0+drA_{k-1}x_0\\
				&=\sum_{l=2}^kA_{k-l}x_{l+1}+A_{k+1}x_0 + A_k x_1 +  A_{k-1} x_2
				=\sum_{l=0}^{k+1}A_{k+1-l}x_l.
		\end{align*}  
This proves the theorem.
\end{proof}
\begin{remark}
It was mentioned in \cite{cirnu} that the usual Fibonacci sequence $\{F_n\}$ cannot be a solution of \eqref{agrelation} with $r=1$. However, if $r=1/2$ , then the Fibonacci numbers $\{F_n\}_{n=0}^{\infty}=\{0, 1, 1, 2, 3, 5, 8,\ldots\}$ become solutions of \eqref{agrelation} with $a=0$ and $d=5/2$ and inital data $x_0=4/5$. In particular, the linear recurrence equation
	\[
		x_{n+1}= \frac{5}{4}x_{n-1}+\frac{5}{4}x_{n-2}+\frac{15}{16}x_{n-3}+\cdots+\frac{5n}{2}x_0, \quad (n\geq 1)
	\]
	with initial condition $x_0=4/5$, has the solution
	\[
		x_n= \frac{\phi^{n-1}-(1-\phi)^{n-1}}{\sqrt{5}}=\!\colon F_{n-1},
	\]
	where $\phi$ denotes the well-known golden ratio, i.e. $\phi=(1+\sqrt{5})/2$.
	
	 Another interesting solution to \eqref{agrelation}, with $r=1/2$ and $a=0$ and $d=9/2$, is the Jacobsthal numbers $\{J_n\}=\{0, 1, 1, 3, 5, 11, 21, \ldots\}$. More precisely, the linear recurrence equation
	\[
		x_{n+1}= \frac{9}{4}x_{n-1}+\frac{9}{4}x_{n-2}+\frac{27}{16}x_{n-3}+\cdots+\frac{9n}{2}x_0, \quad (n\geq 1)
	\]
	with initial condition $x_0=4/9$, has the solution
	\[
		x_n= \frac{2^{n-1}-(-1)^{n-1}}{3}=\!\colon J_{n-1}.
	\]
\end{remark}

\begin{example}
In the following two examples we shall assume $r=1$.

\begin{enumerate}
\item[(1)] The sequence $\{P_n\}_{n=0}^{\infty}$ of Pell numbers $\{0, 1, 2, 5, 12, \ldots\}$ is a solution to \eqref{agrelation}. Indeed, this requires that $x_{n+1} = 2x_{n}-x_{n-1}$ with initial conditions $x_1=a x_0=0$ and $x_2=(a^2+a+d)x_0=1$. Therefore $a+2 = 2$ and $d-a-1 = 1$, giving us $a=0$ and $d=2$, and $x_0=1/2$. Thus, the recurrence
	\[
		x_{n+1}= 2 x_{n-1}+ 4 x_{n-2}+\cdots+ 2nx_0\quad(n\geq 1),
	\]
	with initial condition $x_0=1/2$, has the solution
	\[
		x_n= \frac{\sigma^{n-1}-(1-\sigma)^{n-1}}{\sqrt{2}}=\!\colon P_{n-1},
	\]
	where $\sigma$ denotes the well-known silver ratio, i.e. $\sigma=(1+\sqrt{2})/2$.
\item[(2)] In \cite{behera}, A. Behera and G. K. Panda introduced the concept of \emph{balancing numbers} $n \in \mathbb{N}$ as solutions of the equation
	\[
		1+2+ \cdots + (n -1) = (n+1) + (n+2) + \cdots + (n+r),
	\]
calling $r\in \mathbb{N}$, the \emph{balancer} corresponding to the balancing number $n$. For example $6, 35,$ and $204$ are balancing numbers with balancers $2, 14,$ and $84$, respectively. The sequence $B_n$ of balancing numbers satisfies the relation
	\[
		B_{n+1}=6B_n-B_{n-1}, \quad(n\geq 2),
	\]
	with intial conditions $B_1=1$ and $B_2=6$. It can be easily seen that the balancing numbers are solutions of \eqref{agrelation} with $a=d=4$ and initial condition $x_0=1/4$. More precisely, the linear recurrence relation
	\[
		x_{n+1}= 4x_n+ 8 x_{n-1}+12x_{n-2}+\cdots+ 4(n+1)x_0,	\quad(n\geq 1),
	\]
	with initial condition $x_0=1/4$, has the solution
	\[
		x_n= \frac{\lambda_1^{n-1}-\lambda_2^{n-1}}{\lambda_1-\lambda_2}=\!\colon B_{n-1},
	\]
	where $\lambda_{1,2}$ are roots of the quadratic equation $x^2-6x-1=0$, i.e. $\lambda_{1,2}=3\pm 2\sqrt{2}$.
	
\end{enumerate}
\end{example}

As corollaries to Theorem \eqref{cirnuthm}, we have the following results of C\^{i}rnu \cite{cirnu}. 
\begin{corollary}[\cite{cirnu}, Theorem 2.1]
The numbers $x_n$ are solutions of the linear recurrence relation with the coefficients in arithmetic progression 
	\[
		x_{n+1}=ax_n+(a+d)x_{n-1}+\cdots+(a+nd)x_0 \quad (n\geq 0),
	\]
	with initial data $x_0$ if and only if they are the generalized Fibonacci numbers given by the Binet type formula
	\[
			x_n=\frac{x_0}{\lambda_1-\lambda_2}\left[\left(B-a\lambda_2\right)\lambda_1^{n-1}-\left(B-a\lambda_1\right)\lambda_2^{n-1}\right]\quad (n\geq 1),
	\]
where  $B=a^2+a+d$ and $\lambda_{1,2}=\frac{1}{2}\left(a+2\pm\sqrt{a^2+4d}\right)$.
\end{corollary}

\begin{corollary}[\cite{cirnu}, Theorem 3.1]
The numbers $x_n$ are solutions of the linear recurrence relation with the coefficients in geometric progression
	\[
		x_{n+1}=ax_n+arx_{n-1}+\cdots+ar^nx_0, \quad (n\geq 0)
	\]
	with initial data $x_0$ if and only if they form the geometric progression given by
	\[
		x_n=ax_0(a+q)^{n-1}, \quad(n\geq 1).
	\]
\end{corollary}

\begin{remark}
It was presented in \cite[Corollary 2.2]{cirnu} that the recurrence relation 
	\begin{equation}\label{ex}
		x_{n+1}=x_n+2x_{n-1}+\cdots+nx_1+(n+1)x_0, \quad (n\geq 0),
	\end{equation}
	with the initial data $x_0 = 1$, has the solution
	\begin{equation}\label{rmk}
		x_n=\frac{1}{\sqrt{5}}\left[ \left( \frac{3+\sqrt{5}}{2}\right)^n - \left(\frac{3-\sqrt{5}}{2} \right)^n\right], \quad (n\geq 1).
	\end{equation}
	We point out that $\{x_n \}$ is in fact the sequence of Fibonacci numbers with even indices, i.e. $x_n=F_{2n}$. Moreover, if $x_0$ is replaced by $2$ as initial data of \eqref{rmk}, then we get $x_n=F_{2n+1}$, for $n\geq2$, as solutions to \eqref{ex}.
\end{remark}
\begin{remark}[Convergence Property]
	Using \eqref{cirnuthm}, we can find the limit of the sequence $\{x_{n+\rho}/x_n\}$, where $x_n$ satisfies the relation \eqref{agrelation} and $\rho$ is some positive integer, as $n$ tends to infinity. It is computed as follows:
		\[
			\lim_{n \rightarrow \infty} 	\frac{x_{n+\rho-1}}{x_n}	=  \lim_{n \rightarrow \infty}\frac{\lambda_1^{n+\rho-1}\left[\left(B-a\lambda_2\right)-\left(B-a\lambda_1\right)\left(\frac{\lambda_2}{\lambda_1}\right)^{n+\rho-1}\right]}{\lambda_1^{n-1}\left[\left(B-a\lambda_2\right)-\left(B-a\lambda_1\right)\left(\frac{\lambda_2}{\lambda_1}\right)^{n-1}\right]}
			= \lambda_1^{\rho}.
		\]
\end{remark}

\section{Open problem} The recurrence sequence defined in \eqref{agrelation} can be further generalized in various forms. For instance, we may define the sequence $\{x_n\}$ to satisfy the recurrence relation
	\[
	x_{n+1} = \left\{\begin{array}{cc}
				\sum_{k=0}^n(a+kd)r^kx_{n-k}		&\text{if}\; n\;\text{is even},\\
				\sum_{k=0}^n(b+kc)s^kx_{n-k}		&\text{if}\; n\;\text{is odd},\\
			\end{array}\right.
	\]
	where $a,b,c,d,r$, and $s$ are real numbers with $abrs \neq 0$. This can be further extended into
	\[
	x_{n+1} = \left\{\begin{array}{cc}
				\sum_{k=0}^n(a_1+kd_1)r_1^kx_{n-k}	 &\text{if}\; n\equiv0 \;(\text{mod}\; m),\\
					\vdots&\vdots\\
				\sum_{k=0}^n(a_m+kd_m)r_m^kx_{n-k}	&\text{if}\; n\equiv -1 \;(\text{mod}\; m),
			\end{array}\right.
	\]
	where $a_{i's}, d_{i's}, r_{i's} \in \mathbb{R}$, for all $i= 1, 2, \ldots, m$, with $a_1 a_2 \cdots a_m r_1 r_2 \cdots r_m \neq 0$.

It might be of great interest to study the properties of these sequences (e.g. explicit formula, convergence, etc.).


{\bf Received: February 2, 2015}

\end{document}